\documentclass[12pt,a4paper]{amsart}
\usepackage{tikz}
\usepackage{pdflscape}
\usepackage{amsmath}
\usepackage{amsfonts}
\usepackage{amssymb}
\usepackage{mathtools}
\usepackage[all]{xy}
\usepackage{color}
\usetikzlibrary{arrows}
\usepackage{float}
\usepackage[shortlabels]{enumitem}
\usepackage{ifpdf}
\ifpdf
  \usepackage[colorlinks,
            linkcolor=magenta,       
            anchorcolor=magenta,     
            citecolor=magenta,       
            final,
            hyperindex]{hyperref}
\else
  \usepackage[colorlinks,
            linkcolor=magenta,       
            anchorcolor=magenta,     
            citecolor=magenta,       
            final,
            hyperindex,driverfallback=hypertex]{hyperref}
\fi
\newtheorem{theorem}{Theorem}[section]
\newtheorem{lemma}[theorem]{Lemma}

\newtheorem{prop}[theorem]{Proposition}
\theoremstyle{definition}
\newtheorem{defn}[theorem]{Definition}
\newtheorem{remark}[theorem]{Remark}
\newtheorem{exam}[theorem]{Example}

\makeatletter \@addtoreset{equation}{section} \makeatother
\allowdisplaybreaks

\newcommand{\delete}[1]{}

\newcommand {\Ima}{{\rm Im\,}}
\newcommand {\End}{{\rm End\,}}
\newcommand {\Aut}{{\rm Aut\,}}

\newcommand{\phii}[2]{\phi_{#1,#2}}
\begin{document}
\title[Some characterizations of weak left braces]{Some characterizations of weak left braces}


\author{Shoufeng Wang}
\address{School of Mathematics, Yunnan Normal University, Kunming, Yunnan 650500, China}
\email{wsf1004@163.com}


\begin{abstract} As generalizations of skew left braces, weak left braces were introduced recently by Catino, Mazzotta, Miccoli and Stefanelli  to study  ceratin special degenerate  set-theoretical solutions of the Yang-Baxter equation. In this note, as analogues  of the notions of regular subgroups of holomorph of groups, Gamma functions on groups and affine and semi-affine structures on groups, we propose the notions of good inverse subsemigroups and  Gamma functions associated to Clifford semigroups and affine structures on inverse semigroups, respectively,  by which weak left braces are characterized. Moreover, symmetric, $\lambda$-homomorphic  and $\lambda$-anti-homomorphic weak left braces are introduced and the algebraic structures of these weak left braces are given.
\end{abstract}
\makeatletter
\@namedef{subjclassname@2020}{\textup{2020} Mathematics Subject Classification}
\makeatother
\subjclass[2020]{20M18, 16T25, 16Y99}
\keywords{Inverse  semigroup, Clifford semigroup, weak left brace, Gamma function}

\maketitle



\vspace{-0.8cm}

\section{Introduction}

The Yang-Baxter equation first appeared in Yang's study of exact solutions of many-body problems in \cite{Yang} and Baxter's work of the eight-vertex model in \cite{Baxter}. The equation is closely related to many fields in mathematics and physics. Due to the complexity in solving the Yang-Baxter
equation, Drinfeld suggested to study set-theoretical solutions of the Yang-Baxter equation in \cite{Drinfeld}.
To construct set-theoretical solutions of the Yang-Baxter equation by the view of algebra, several kinds of algebraic structures have been introduced in literature. In particular,  Rump introduced left braces in \cite{Rump2} as a generalization of Jacobson radical rings, and a few years later, Ced$\acute{\rm o}$, Jespers  and Okni$\acute{\rm n}$ski \cite{Cedo-Jespers-Okninski} reformulated
Rump's definition of left braces. Guarnieri and Vendramin \cite{Guarnieri-Vendramin} introduced skew left braces as generalization of left braces. In particular, they showed that skew left braces are connected closely with regular subgroups of holomorphs of groups and proved that (skew) left braces can provide (involutive) non-degenerate  set-theoretical solutions of the Yang-Baxter equation. Campedel, Caranti and Del Corso \cite{Campedel} investigated skew left braces by using so-called Gamma functions on groups introduced there. Left braces and  left semi-braces were characterized by affine and semi-affine structures on groups in \cite{Rump3} and \cite{Stefanelli1}, respectively. Moreover, Rump \cite{Rump4} obtained instances of left braces in terms of affine structures of decomposable solvable groups.
On the other hand, several special classes skew left braces, such as bi-skew left braces (i.e. symmetric skew left braces),  $\lambda$-homomorphic  and $\lambda$-anti-homomorphic skew left braces, have been introduced and studied (see \cite{Bardakov1,Bardakov2,Caranti,Childs} and the reference therein).

Recently, Catino, Mazzotta, Miccoli and Stefanelli \cite{Catino-Mazzotta-Miccoli-Stefanelli}, and Catino, Mazzotta and Stefanelli \cite{c1} introduced weak left braces and dual weak left braces by using inverse semigroups and Clifford semigroups, respectively, and proved that weak left braces and dual weak left braces can provide ceratin special degenerate  set-theoretical solutions of the Yang-Baxter equation. More recent information on  weak left braces  and their generalizations  can be found in  \cite{Catino-Mazzotta-Stefanelli3,Gong-Wang,Liu-Wang,Mazzotta}. In particular, the text \cite{Catino-Mazzotta-Stefanelli3} gave a construction method for dual weak left braces.

The aim of this note is to continue the investigation of weak left braces by using the existing results.  We obtain some characterizations of weak left braces,  dual weak left braces and some special classes of them. In Section 2, some necessary preliminaries on inverse semigroups, Clifford semigroups and weak left braces are presented.  In Section 3, we introduce the notions of good inverse subsemigroups and Gamma functions associated to Clifford semigroups and affine structures on inverse semigroups by which weak left braces are characterized. In Section 4, we introduce symmetric, $\lambda$-homomorphic  and $\lambda$-anti-homomorphic weak left braces, respectively, and show that these   weak left braces are all dual weak left braces. The results obtained in the present paper can be regarded as the enrichment and extension of the corresponding results in  the texts \cite{Bardakov1,Bardakov2,Campedel,Caranti,Guarnieri-Vendramin,Rump3,Stefanelli1}.

 \section{Preliminarie}
In this section, we recall some notions and known results on inverse semigroups, Clifford semigroups and weak left braces.
Let $(S, +)$ be a semigroup. We denote the set of all idempotents in $(S, +)$ by $E(S,+)$ and the center of $(S, +)$ by $C(S,+)$, respectively.  That is,
$$E(S, +)=\{x\in S\mid x+x=x\},\,\, C(S,+)=\{x\in S\mid x+a=a+x \mbox{ for all } a\in S\}.$$
According to \cite{A10}, a semigroup $(S,+)$ is called an {\em inverse semigroup} if for all $x\in S$, there exists a unique element $-a\in S$ such that $a=a-a+a$ and $-a+a-a=-a$. Such an element $-a$ is called {\em the  von Neumann inverse} of $a$ in $S$. Let $(S, +)$ be an inverse semigroup. It is well known that $-(a+b)=-b-a$ and $-(-a)=a$ for all $a,b\in S$. Moreover, $E(S,+)$ is a commutative subsemigroup of $(S, +)$, and $$E(S, +)=\{-x+x\mid x\in S\}=\{x-x\mid x\in S\} \mbox{ and }-e=e \mbox{ for all }e\in E(S,+).$$

 The following lemma  gives a characterization of inverse semigroups. Recall that a semigroup $(S,+)$ is {\em regular} if for each $a\in S$, there exists $a'\in S$ such that $a+a'+a=a$.
\begin{lemma}[\cite{A10}]\label{inverse} A semigroup $(S,+)$ is an inverse semigroup if and only if $(S,+)$ is  regular and $e+f=f+e$ for all $e,f\in E(S, +)$.
\end{lemma}

An inverse semigroup $(S, +)$ is called a {\em Clifford semigroup} if $-a+a=a-a$ for all $a\in S$.
In this case, we denote $a^0=-a+a=a-a$ for all $a\in S$.
The following lemma  gives a characterization of Clifford semigroups.

\begin{lemma}[\cite{A10}]\label{Clifford}A semigroup $(S, +)$ is a  Clifford semigroup if and only if $(S,+)$ is a regular semigroup and $E(S, +) \subseteq C(S,+)$.
\end{lemma}
Let $(S, +)$ and $(S', +')$ be two Clifford semigroups. A map (respectively, bijection) $\phi: S\rightarrow S'$ is called a
{\em   semigroup homomorphism} (respectively, {\em   semigroup isomorphism}) from $(S, +)$ to $(S', +')$ if $\phi(x+y)=\phi(x)+'\phi(y)$ for all $x,y\in S$.
A {\em   semigroup homomorphism} (respectively,  {\em   semigroup isomorphism}) from $(S, +)$ to itself is called a {\em   semigroup endomorphism} (respectively,  {\em   semigroup automorphism}) on $(S, +)$. The set of all {\em   semigroup endomorphisms} (respectively,  {\em   semigroup automorphisms}) on $(S, +)$ are denoted by $\End (S,+)$  and $\Aut (S,+)$, respectively. Obviously, $\End (S,+)$  and $\Aut (S,+)$ forms  a semigroup and a group with respect to the compositions of maps, respectively.
In the sequel we shall use the following lemma throughout the paper without further mention.
\begin{lemma}[\cite{A10}]\label{basicproperties}\label{jichu} Let $(S, +)$ and $(S',+')$ be two Clifford semigroups, $a,b\in S$, $e\in E(S, +)$ and $\phi: (S, +)\rightarrow (S',+')$ be a semigroup homomorphism. Then
$$a=a-a+a,\,\, -a+a-a=-a,\,\,\, -(a+b)=-b-a,\,\, -(-a)=a,\,\, e+a=a+e,  $$$$e^0=e=-e,\,\,(-a)^0=-a^0=a^{00}=a^0=-a+a=a-a,\,\,a^0+a=a+a^0=a,$$
$$a^0+a^0=a^0, a^0 +b=b+a^0, (a-b)^0=(a+b)^0=a^0+b^0,\,\ \phi(-a)=-\phi(a),\,\,\, \phi(a^0)=\phi(a)^0.$$
\end{lemma}

From \cite{Catino-Mazzotta-Miccoli-Stefanelli},  a triple $(S, +, \cdot)$ is called a {\em weak  left brace} if $(S, +)$ and $(S, \cdot)$ are  inverse  semigroups and the following axioms hold:
\begin{equation}\label{dengshi}x(y+z)=xy-x+xz,\,\,\,\,\,\, x x^{-1}=-x+x,
\end{equation} where $x^{-1}$ is the inverse of $x$ in $(S, \cdot)$ for all $x\in S$.
By the second axiom,  we have $E(S, +)=E(S, \cdot)$.  We also denote this set by $E(S)$ sometimes in this case.  From \cite{c1}, a weak  left brace $(S, +, \cdot)$ is called a {\em dual weak left brace} if $(S, \cdot)$  is a Clifford semigroup. From \cite{Guarnieri-Vendramin}, a dual weak left brace $(S, +, \cdot)$ is called a {\em skew left brace} if
 $(S, +)$ and $(S, \cdot)$ are groups. For weak left braces, we have the following key lemmas by Lemma \ref{basicproperties}.
\begin{lemma}[Theorem 8 in \cite{Catino-Mazzotta-Miccoli-Stefanelli} and Lemma 2 in \cite{Catino-Mazzotta-Stefanelli3}]\label{wang5}
Let $(S, +, \cdot)$ be a  weak left  brace and  $a,b\in T, e\in E(S)$.
Then $(S,+)$ is a Clifford semigroup and $$e  a=e+a=a+e, a^0+a=a+a^0=a,$$
$$a^0+b^0=(a+b)^0, (-a)^0=a^0=a^0-a^0=-a^0=a-a=-a+a.$$
In particular, if $(S, +, \circ)$ is a dual weak left brace,  then $x^0=x-x=-x+x=x  x^{-1}=x^{-1}  x$ for all $x\in S$, and
$$a= a^0  a=a  a^0,\,a  e=a+e=e+a,\,(a  b)^0=a^0  b^0=a^0+b^0,  (a^{-1})^0=a^0.$$
\end{lemma}
By the comments after Definition 5 and Proposition 6, Lemma 1, Proposition 7, Theorem 8 in \cite{Catino-Mazzotta-Miccoli-Stefanelli}, we have the following lemma.
\begin{lemma}[\cite{Catino-Mazzotta-Miccoli-Stefanelli}]\label{wangkangcccc}
Let $(T, +, \cdot)$ be a  weak left  brace. Define $\lambda_a: T\rightarrow T, b\mapsto -a+ab$ for all $a\in T$.  Then $\lambda_a\in \End (T,+)$.
Furthermore, define $\lambda: (T, \cdot)\rightarrow \End (T,+), \,\, a\mapsto \lambda_a.$ Then $\lambda$ is a semigroup homomorphism.
Moreover, we have the following result:
\begin{itemize}
\item[(1)] $0$ is the identity of $(T, +)$ if and only if $0$ is the identity of $(T, \cdot)$.
\item[(2)]   $ab=a+\lambda_a(b)$ and $\lambda_a(a^{-1})=-a$ for all $a,b\in T$.
\end{itemize}
\end{lemma}
  Let $(T,+)$ be a  Clifford semigroup and  $a,b\in T$. Denote $H_{a}=\{x\in T\mid x^{0}=a^{0}\}$. Then it is easy to see that $(H_{a},+)$ is a subgroup of $(T,+)$ and $x^{0}$ is the identity in $(H_{a},+)$ for all $x\in H_a$ by Lemma \ref{jichu}. Obviously, $H_{a}=H_{b}$ if and only if $a^{0}=b^{0}$.
\begin{lemma}\label{jubutonggou}
Let $(T, +, \cdot)$ be a dual weak left  brace and $a\in T$. Then $\lambda_a|_{H_a}\in \Aut(H_a,+)$.
\end{lemma}
\begin{proof}
Let $x\in H_a$. Then $x^0=a^0$ and $$(\lambda_a(x))^0=(-a+ax)^0=(-a)^0+(ax)^0=a^0+a^0x^0=a^0+a^0a^0=a^0+a^0=a^0,$$ and so $\lambda_a(x)\in H_a$. This implies that $\lambda_a(H_a)\subseteq Ha$.
Since $\lambda_a\in \End(T,+)$ by Lemma \ref{wangkangcccc}, we have $\lambda_a|_{H_a}\in \End(H_a, +)$.  In view of Lemma \ref{wang5}, we have $(a^{-1})^0=a^0$, and so $H_{a^{-1}}=H_a$. Therefore $\lambda_{a^{-1}}|_{H_a}\in \End(H_a, +)$. For every $x\in H_a$, it follows that   $$\lambda_a|_{H_a}\lambda_{a^{-1}}|_{H_a}(x)=\lambda_{aa^{-1}}(x)=\lambda_{a^0}(x)=-a^0+a^0x=a^0+a^0+x=x^0+x^0+x=x$$ by the facts that  $\lambda_a\in \End(T,+)$  and $x^0=a^0$ and Lemmas \ref{wang5} and \ref{wangkangcccc}. This implies   $\lambda_a|_{H_a}\lambda_{a^{-1}}|_{H_a}={\rm id}_{H_a}$. Dually, $\lambda_{a^{-1}}|_{H_a}\lambda_a|_{H_a}={\rm id}_{H_a}$. Thus $\lambda_a|_{H_a}\in \Aut(H_a,+)$.
\end{proof}

\section{Some characterizations of weak left braces}
In this section, we shall give  some characterizations of weak left braces.  We first give an analogue of \cite[Proposition 1.9]{Guarnieri-Vendramin} which is also  a converse of Lemma \ref{wangkangcccc} in some sense.
\begin{prop}
Let $(S, +)$ and $(S, \cdot)$ be two inverse semigroups. Then $S$ satisfies the axiom
\begin{equation}\label{weak1}
x(y+z)=xy-x+xz
\end{equation}
if and only if both
$$\lambda_x: (S,+) \rightarrow (S,+),   x \mapsto -x+xy\,\,\, \mbox{ and } \lambda: (S,\cdot) \rightarrow \End(S,+),   \lambda \mapsto \lambda_x$$
are semigroup homomorphisms for all $x$ in $S$.
\end{prop}
\begin{proof}We first observe that the conditions that $\lambda_x$ is a semigroup homomorphism for all $x$,  and $\lambda$ is a semigroup homomorphism
are equivalent to the axioms
\begin{equation}\label{weak2}
-x+x(y+z)=-x+xy-x+xz
\end{equation}
and
\begin{equation}\label{weak3}
-xy+xyz=-x+x(-y+yz),
\end{equation}
respectively.
Let $x,y,z\in S$. We first show the necessity. On one hand, we have (\ref{weak2}) quickly by (\ref{weak1}).
On the other hand, we obtain that
$$xy-x+x(-y)-x+xy=x(y-y)-x+xy=x(y-y+y)=xy,$$$$-x+x(-y)-x+xy-x+x(-y)-x=-x+x(-y+y)-x+x(-y)-x$$
$$-x+x(-y+y-y)-x=-x+x(-y)-x.$$ This shows that $-xy=-x+x(-y)-x$ as $(S,+)$ is an inverse semigroup, and so
$$-x+x(-y+yz)=-x+x(-y)-x+xyz=-xy+xyz.$$ This gives (\ref{weak3}).
To show the sufficiency, let $z=y^{-1}y$ in (\ref{weak3}). Then we have
$$-xy+xy=-xy+xyy^{-1}y=-x+x(-y+yy^{-1}y)$$$$=-x+x(-y+y)=-x+x(-y)-x+xy,$$
$$xy=xy-xy+xy=xy-x+x(-y)-x+xy.$$ On the other hand, by (\ref{weak2}),
$$-x+x(-y)-x+xy-x+x(-y)-x=-x+x(-y+y)-x+x(-y)-x$$$$=-x+x(-y+y-y)-x=-x+x(-y)-x.$$
This shows that $-xy=-x+x(-y)-x$ as $(S,+)$ is an inverse semigroup, and so
$xy=x-x(-y)+x$ and $$x-x+xy=x-x+x-x(-y)+x=x-x(-y)+x=xy$$ for all $x, y\in S$. Thus
$$x(y+z)=x-x+x(y+z)=x-x+xy-x+xz=xy-x
=xz.$$This gives (\ref{weak1}).
\end{proof}
\begin{remark}
Let $S=\{0,1\}$ and define  two binary operations $+$ and $\cdot$ as follows: $0+0=0+1=1+0=0, 1+1=1$ and $0\cdot 1=1\cdot 0=1\cdot 1=1, 0\cdot 0=0$. Then one can check that $(S,+)$ and $(S, \cdot)$ are idempotent inverse semigroups and (\ref{weak2}) is satisfied.
Since $0\cdot (1+1)=0\cdot 1=1$ and $0\cdot 1-0+0\cdot 1=1-0+1=1$, it follows that (\ref{weak1}) fails.
On the other hand, let $S=\{0,1,2\}$. Define two binary operations $+$ and $\cdot$ as follows:
\begin{center}
\begin{tabular}{r|rrr}
+ & 0 & 1 & 2\\
\hline
    0 & 0 & 0 & 2 \\
    1 & 0 & 1 & 2 \\
    2 & 2 & 2 & 2
\end{tabular} \hspace{.5cm}
\begin{tabular}{r|rrr}
$\cdot$ & 0 & 1 & 2\\
\hline
    0 & 0 & 2 & 2 \\
    1 & 2 & 1 & 2 \\
    2 & 2 & 2 & 2
\end{tabular}
\end{center}
Then one can check that $(S,+)$ and $(S, \cdot)$ are idempotent inverse semigroups and (\ref{weak3}) is satisfied.
Since $0\cdot (0+1)=0\cdot 0=0$ and $0\cdot 0-0+0\cdot 1=0-0+2=0+2=2$, it follows that (\ref{weak1}) fails.
Finally, we observe that (\ref{weak1})--(\ref{weak3}) together do not imply the axiom $xx^{-1}=-x+x$. For example,
let $S=\{0,1\}$ and define  two binary operations $+$ and $\cdot$ as follows: $0+0=1, 0+1=1+0=0, 1+1=1$ and $0\cdot 1=1\cdot 0=0\cdot 0=0, 1\cdot 1=1$. Then one can check that $(S,+)$ and $(S, \cdot)$ are idempotent inverse semigroups and (\ref{weak1})--(\ref{weak3}) are satisfied.
Since $0\cdot 0^{-1}=0\cdot 0=0$ and $-0+0=0+0=1$, it follows that the involved axiom fails.
\end{remark}

We next characterize weak left braces by using so-called good inverse subsemiroups. To this aim, we need some necessary preliminaries.
Let $(S, +)$ be a Clifford semigroup. Define a multiplication on $\End(S,+)\times (S,+)$ as follows:
$$(f,x)(g,y)=(fg, x+f(y)) \mbox{ for all } (f,x)(g,y)\in \End(S,+)\times (S,+).$$
One can check   that $\End(S,+)\times (S,+)$  forms a semigroup with respect to the above multiplication, and is denoted by $\End(S,+)\rtimes (S,+)$.
\begin{defn}
Let $(S, +)$ be a Clifford semigroup. An inverse (respectively, a Clifford) subsemigroup $H$ of $\End(S,+)\rtimes S$ is called {\em good} if  the following conditions hold:
\begin{itemize}
\item[\rm (G1)] $\pi_2|_H: H\rightarrow S, (f,x)\mapsto x$ is a bijection.
\item[\rm (G2)] If $(f,x)\in H$ and $(f,x)^{-1}=(g,y)$, then $(f, f(-y))\in H$.
\item[\rm (G3)] If $(f,x)\in H$, then $-x+x+f(y)=f(y)$ for all $y\in S$.
\item[\rm (G4)] If $(S,+)$ has an identity $0$, then  $f(0)=-x+x$ for all $(f,x)\in H$.
\end{itemize}
\end{defn}
\begin{prop}\label{good1}
Let $(S, +)$ be a Clifford semigroup and $H$ a good  inverse  (respectively,  Clifford)  subsemigroup  of $\End(S,+)\rtimes (S,+)$. Define a multiplication $\circ$ on $S$ as follows:
\begin{equation}\label{weak4}
a\circ b=a+f(b) \mbox{ for all } a, b\in S,
\end{equation}
where $(\pi_2|_H)^{-1}(a)=(f,a)\in H$. Then $(S, +, \circ)$ forms a weak left brace (respectively, dual weak left brace) and $H\cong (S, \circ)$. We denote the  weak left brace $(S, +, \circ)$ by $\mathcal{B}(H)$.
\end{prop}
\begin{proof}
Since $(\pi_2|_H)^{-1}$ is a bijection, we can transfer the multiplication of $H$ to a multiplication $\circ$ on $S$ as follows:
$$a\circ b=(\pi_2|_H)((\pi_2|_H)^{-1}(a) (\pi_2|_H)^{-1}(b))=a+f(b), $$ where $(\pi_2|_H)^{-1}(a)=(f,a)\in H$. Then $(S, \circ)$ forms an inverse semigroup and $\pi_2|_H$ is an isomorphism from $H$ onto $(S, \circ)$.
Let $a,b,c\in S$. Then $$a\circ (b+c)=a+f(b+c)=a+f(b)+f(c)$$$$=a-a+a+f(b)+f(c)=a+f(b)-a+a +f(c)=a\circ b -a+a\circ c.$$ Let $(\pi_2|_H)^{-1}(a)=(f,a)\in H$ and $(f,a)^{-1}=(g,b)$. Then the inverse of $a$ in $(S,\circ)$ is $a^{-1}=b$ and
$(f,a)(g,b)(f,a)=(f,a),\,\, (g,b)(f,a)(g,b)=(g,b).$ This implies that $f=fgf, a+f(b)+fg(a)=a$ and $gfg=g, b+g(a)+gf(b)=b$. So $b-gf(b)+gf(b)=b$.
By (G2), $(f, f(-b))\in H$. Observe that
$$(g,b)(f,f(-b))(g,b)=(gfg, b+gf(-b)+gf(b))=(g,b),$$
$$(f,f(-b))(g,b)(f,f(-b))=(fgf, f(-b)+f(b)+fgf(-b))=(f,f(-b)).$$
Since $H$ is an inverse semigroup, we have $(f,a)^{-1}=(g,b)=(f,f(-b))^{-1}$, and so $(f,a)=(f,f(-b))$.  This yields that $a=f(-b)=-f(b)$ and $f(b)=-a$.
So $$a\circ a^{-1}=a\circ b=a+f(b)=a-a=-a+a.$$ Thus, $(S,+, \circ)$ is a weak left brace. If $H$ is a Clifford subsemigroup, then $(S, \circ)$  is also a Clifford
semigroup, and so $(S,+, \circ)$ is a dual weak left brace.
\end{proof}
\begin{remark}\label{good3}
Let $(S, +)$ be a Clifford semigroup and $H$ a good inverse subsemigroup of $\End(S,+)\rtimes (S, +)$. By Proposition \ref{good1} and its proof, $(S,+,\circ)$ is a weak left brace and $\pi_2|_H$ is a semigroup isomorphism from $H$ onto $(S, \circ)$. By Lemma \ref{wang5}, we have $E(S, +)=E(S, \circ)$ and $e\circ a=e+a$ for all $e\in E(S, +)=E(S, \circ)$ and $a\in S$. Thus
$$x\in E(S, +)\Longleftrightarrow x\in E(S, \circ)\Longleftrightarrow (\pi_2|_H)^{-1}(x)\in E(H),$$
$$(\pi_2|_H)^{-1}(e)(\pi_2|_H)^{-1}(x)=(\pi_2|_H)^{-1}(e+x)$$ for all $e\in E(S,+)$ and $x\in S$.
\end{remark}

Let $(S, +, \cdot)$ be a weak left brace. Then $(S,+)$ is a Clifford semigroup and $(S, \cdot)$ is an inverse semigroup by Lemma \ref{wang5}. Consider the semigroup $\End(S,+)\rtimes (S,+)$. By Lemma \ref{wangkangcccc},
we have $(\lambda_a,a)\in \End(S,+)\rtimes (S,+)$, where $\lambda_a$ is defined as in Lemma \ref{wangkangcccc} for all $a\in S$.
\begin{prop}\label{good2}Let $(S, +, \cdot)$ be a weak (respectively, dual weak) left brace.
With the above notion, ${\mathcal S}(S)=\{(\lambda_a, a)\mid a\in S\}$ is a good inverse (respectively, Clifford) subsemigroup of $\End(S,+)\rtimes (S,+)$.
\end{prop}
\begin{proof}
Let $a,b\in S$. Then $(\lambda_a, a)(\lambda_b,b)=(\lambda_a\lambda_b, a+\lambda_a(b))=(\lambda_{ab}, ab)$ by Lemma \ref{wangkangcccc}.
This shows that ${\mathcal S}(S)$ is a subsemigroup of $\End(S,+)\rtimes (S,+)$. Obviously, $$\psi: {\mathcal S}(S)\rightarrow (S, \cdot), (\lambda_a, a)\mapsto a$$ is an isomorphism from ${\mathcal S}(S)$ onto $(S, \cdot)$. So ${\mathcal S}(S)$ is an inverse subsemigroup and $(\lambda_a,a)^{-1}=(\lambda_{a^{-1}}, a^{-1})$. Evidently,
$\pi_2|_{{\mathcal S}(S)}: {\mathcal S}(S)\rightarrow S, (\lambda_a,a)\mapsto a$ is a bijection from ${\mathcal S}(S)$ to $S$, and so (G1) holds. Let $(\lambda_a,a)\in {\mathcal S}(S)$. Then $(\lambda_a,a)^{-1}=(\lambda_{a^{-1}}, a^{-1})$. Moreover, $\lambda_{a^{-1}}(-a)=a^{-1}$ by Lemma \ref{wangkangcccc}. So $(\lambda_a, \lambda_{a^{-1}}(-a))=(\lambda_{a^{-1}}, a^{-1})\in H$. This gives (G2). Finally, let $(\lambda_a,a)\in {\mathcal S}(S)$ and $b\in S$. Then $$-a+a+\lambda_a(b)=-a+a+(-a+ab)=-a+ab=\lambda(b).$$ This proves (G3).  If $(S,+)$ has an identity $0$, then   $0$ is also the identity of $(S, \cdot)$  by Lemma \ref{wangkangcccc}.  This gives that $\lambda_a(0)=-a+a\cdot 0=-a+a$ for all $a\in S$.
Thus ${\mathcal S}(S)$ is a good inverse subsemigroup of  $\End(S,+)\rtimes (S,+)$.  In particular,  if $(S, +, \cdot)$ is a  dual weak left brace, then $(S,\cdot)$ is a Clifford semigroup, and so ${\mathcal S}(S)$ is a good  Clifford  subsemigroup.
\end{proof}
The following is an analogue of \cite[Theorem 4.2, Proposition 4.3]{Guarnieri-Vendramin}.
\begin{prop}\label{good4}
Let $(S,+)$ be a Clifford semigroup and denote
$$\mathbb{WB}^+=\{(S,+,\cdot)\mid (S,+,\cdot) \mbox{ forms a weak left brace}\},$$
$$\mathbb{GIS}=\{H\mid H \mbox{ is a good inverse subsemigroup of }\End(S,+)\rtimes (S,+)\}$$
Then
${\mathcal S}: \mathbb{WB}^+\rightarrow \mathbb{GIS},\, (S,+,\cdot)\mapsto {\mathcal S}(S)$ and ${\mathcal B}:\mathbb{GIS}\rightarrow\mathbb{WB}^+,\, H\mapsto  {\mathcal B}(H)$ are mutually inverse bijections, and  dual weak left braces correspond to good Clifford subsemigroups. Moreover, if  $(S,+)$ has an identity $0$, then isomorphic weak left brace structures in $\mathbb{WB}^+$ correspond to conjugate inverse subsemigroups of $\End(S,+)\times (S,+)$.
\end{prop}
\begin{proof}By Propositions \ref{good1} and \ref{good2}, the above maps are well-defined and  dual  weak left braces correspond to good Clifford subsemigroups.

Let $(S, +, \cdot)\in \mathbb{WB}^+$ and denote ${\mathcal B}{\mathcal S}(S)=(S, +, \circ)$. Then by the proofs of Propositions \ref{good1} and \ref{good2}, we have $a\circ b=a+\lambda_a(b)=a\cdot b$ for all $a,b\in S$. This shows that $\circ$ and $\cdot$ are the same. Thus ${\mathcal B}{\mathcal S}={\rm id}_{\mathbb{WB}^+}$. Conversely, let $H\in \mathbb{GIS}$ and ${\mathcal B}(H)=(S, +, \circ)$. Then  ${\mathcal S}{\mathcal B}(H)=\{(\lambda_a,a)\mid a\in S\}$ and $\lambda_a(b)=-a+a\circ b$ for all $a, b\in S$. If  $(\lambda_a,a)\in {\mathcal S}{\mathcal B}(H), a\in S$, then for all $b\in S$, by (G3) we have $\lambda_a(b)=-a+a\circ b=-a+a+f(b)=f(b)$, where $(\pi_2|_H)^{-1}(a)=(f,a)\in H$. This shows that $(\lambda_a,a)=(f,a)\in H$. On the other hand, let $(f,a)\in H$. Then by (G3), $f(b)=-a+a+f(b)=-a+a\circ b=\lambda_a(b)$ for all $b\in S$, and so $(f,a)=(\lambda_a, a)\in {\mathcal S}{\mathcal B}(H).$ Thus ${\mathcal S}{\mathcal B}={\rm id}_{\mathbb{GIS}}$.

Assume that $0$ is the identity of $(S,+)$. Then by  by Lemma \ref{wangkangcccc}, $0$ is the identity of $(S, \cdot)$ for each $(S, +, \cdot)\in \mathbb{WB}^+$. Let $(S, +, \cdot), (S, +, \circ)\in \mathbb{WB}^+$ and $\phi: (S, +, \cdot)\rightarrow (S, +, \circ)$ be a weak left brace isomorphism.
Then $(\phi, 0), (\phi^{-1}, 0)\in \End(S,+)\times (S,+)$. Denote the good inverse subsemigroups associated $(S, +, \cdot)$ and $(S, +, \circ)$ by $H_\cdot=\{(\lambda_a, a)\mid a\in S\}$ and $H_\circ=\{(\mu_a, a)\mid a\in S\}$, respectively.  Observe that $$(\phi, 0)(\lambda_a, a)(\phi^{-1}, 0)=(\phi\lambda\phi^{-1}, 0+\phi(a)+\phi\lambda_a(0))=(\phi\lambda\phi^{-1}, 0+\phi(a)+\phi(-a+a\cdot 0))$$$$=(\phi\lambda\phi^{-1}, \phi(a)+\phi(-a+a))=(\phi\lambda\phi^{-1}, \phi(a-a+a))=(\phi\lambda\phi^{-1}, \phi(a)),$$
$$\phi\lambda\phi^{-1}(b)=\phi(-a+a\cdot\phi^{-1}(b))=\phi(-a)+\phi(a)\circ \phi\phi^{-1}(b)=-\phi(a)+\phi(a)\circ  b=\mu_{\phi(a)}(b)$$
for all $a,b\in S$. This shows that $(\phi, 0)(\lambda_a, a)(\phi^{-1}, 0)=(\mu_{\phi(a)}, \phi(a))$ for all $a\in S$. Thus $(\phi,0)H_\cdot (\phi^{-1},0)=H_{\circ}$.  Conversely, assume that $H,K\in \mathbb{GIS}$ and $(\psi,0)H  (\psi^{-1},0)=K$ for some $\psi\in \Aut(S, +)$. Observe that $(\psi, 0), (\psi^{-1}, 0)\in \End(S,+)\rtimes (S,+)$ in this case. Let $a,b\in S$ and $(\pi_2|_H)^{-1}(a)=(f,a), (\pi_2|_K)^{-1}(a)=(g,a)\in K$. Then $a\bullet b=a+f(b)$ and $a\circ b=a+g(b)$. Since $(\psi,0)H  (\psi^{-1},0)=K$, we have $$K\ni=(\psi, 0)(f,a)(\psi^{-1}, 0)=(\psi f\psi^{-1}, 0+\psi(a)+\psi f(0))$$$$\overset{\rm(G4)}{=}(\psi f\psi^{-1}, \psi(a)+\psi(-a+a))=(\psi f\psi^{-1}, \psi(a-a+a))=(\psi f\psi^{-1}, \psi(a)).$$
This gives that $$\psi(a)\circ \psi(b)=\psi(a)+\psi f\psi^{-1}(\psi(b))=\psi(a)+\psi(f(b))=\psi(a+f(b))=\psi(a\cdot b).$$
Thus $\psi$ is an isomorphism from ${\mathcal B}(H)$ onto ${\mathcal B}(K)$.
\end{proof}

Thirdly, we characterize weak left braces by using  Gamma functions on Clifford semigroups. Inspired by the Gamma functions on groups given in \cite{Campedel}, we can formulate the notion of Gamma functions for Clifford semigroups as follows.
\begin{defn}\label{xx}
Let $(S,+)$ be a Clifford semigroup and $\gamma: (S,+)\rightarrow \End(S,+),\, x\mapsto \gamma_x$ be a map. Then $\gamma$ is called a {\em Gamma function} on $(S,+)$ if for all $x,y \in S$, the following conditions hold:
\begin{itemize}
\item[\rm (F1)] $\gamma_x\gamma_y=\gamma_{x+\gamma_{x}(y)},\, \gamma_{x^0}(x)=x$ and $x^0+\gamma_{x}(y)=\gamma_{x}(y)$.
\item[\rm (F2)] If $\gamma_x(x)=x^0$, then $x=x^0$.
\item[\rm (F3)] There exists $x^{-1}\in S$ such that $\gamma_x(x^{-1})=-x$ and  $\gamma_{x^{-1}}(x)=-x^{-1}$.
\item[\rm (F4)] If $e,f\in E(S,+)$, then  $\gamma_e(f)=e+f$.
\end{itemize}
\end{defn}
\begin{defn}\label{xxx}
Let $(S,+)$ be a Clifford semigroup and $\gamma: (S,+)\rightarrow \End(S,+),\, x\mapsto \gamma_x$ be a map. Denote  $H_x=\{a\in S\mid a^0=x^0\}$ for every $x\in S$. Then $\gamma$ is called a {\em dual Gamma function} on $(S,+)$ if for all $x,y \in S$, the following conditions hold:
\begin{itemize}
\item[\rm(D1)] $\gamma_x|_{H_x}$ is an automorphism of the subgroup $(H_x, +)$ of $(S,+)$.
\item[\rm (D2)] $\gamma_x\gamma_y=\gamma_{x+\gamma_{x}(y)}$ and $x^0+\gamma_{x}(y)=\gamma_{x}(y)$.
\item[\rm (D3)] If $e,f\in E(S,+)$, then  $\gamma_e(f)=e+f$.
\end{itemize}
\end{defn}
\begin{prop}\label{good7}Let $(S,+)$ be a Clifford semigroup and $\gamma$ be a   dual Gamma function on $(S,+)$. Then $\gamma$ is a Gamma function on $(S,+)$.
\end{prop}
\begin{proof}Let $x,y\in S$. Then $x^0\in H_{x^0}=H_x$, and so $\gamma_{x^0}|_{H_x}\in \Aut(H_x,+)$ by (D1), whence $\gamma_{x^0}(x^0)=\gamma_{x^0}|_{H_x}(x^0)=x^0$. By (D2), we have $\gamma_{x^0}\gamma_{x^0}=\gamma_{x^0+\gamma_{x^0}(x^0)}=\gamma_{x^0+x^0}=\gamma_{x^0}$, whence
 $\gamma_{x^0}|_{H_x}\gamma_{x^0}|_{H_x}=\gamma_{x^0}|_{H_x}\in \Aut(H_x,+)$. This implies that $\gamma_{x^0}|_{H_x}={\rm id}_{H_x}$, and hence $$\gamma_{x^0}(x)=\gamma_{x^0}|_{H_x}(x)={\rm id}_{H_x}(x)=x.$$ This together with (D2) gives (F1).  Since $\gamma_x\in \Aut(H_x,+)$ by (D1) and $x^0\in H_x$, we have $\gamma_x(x^0)=x^0$.  If $\gamma_x(x)=x^0$, then $x=x^0$ as $\gamma_x$ is injective. Thus (F2) is true. Finally, since $-x\in H_x$ by Lemma \ref{wang5} and  $\gamma_x\in \Aut(H_x,+)$, we have $\gamma^{-1}_x(-x)\in H_x$. Denote $x^{-1}=\gamma^{-1}_x(-x)$. Then $H_x=H_{x^{-1}}$ and $\gamma_{x^{-1}}\in \Aut(H_x,+)$.
Obviously, $\gamma_x(x^{-1})=\gamma_x(\gamma^{-1}_x(-x))=-x$.  Observe that
$$\gamma_x(\gamma_{x^{-1}}(x))=(\gamma_x\gamma_{x^{-1}})(x)\overset{\rm(D2)}{=}\gamma_{x+\gamma_{x}(x^{-1})}(x)=\gamma_{x-x}(x)=\gamma_{x^0}(x)$$$$=x=-(-x)
=-(\gamma_x\gamma^{-1}_x(-x))=-\gamma_x(\gamma^{-1}_x(-x))=\gamma_x(-\gamma^{-1}_x(-x)),$$ it follows that $\gamma_{x^{-1}}(x)=-\gamma^{-1}_x(-x)=-x^{-1}$ as $\gamma_x$ is injective. Thus (F3) holds.
\end{proof}
\begin{prop}\label{good5}Let $(S,+,\cdot)$ be a weak left brace. Then
$\lambda: (S, +)\rightarrow \End(S,+), x\mapsto \lambda_x $
is a Gamma function on $(S,+)$, where $\lambda_x(y)=-x+xy$ for all $x,y\in S$. In this case, we denote $\lambda$ by ${\mathcal G}(S)$. In particular, if $(S,+,\cdot)$ is also a dual weak left brace, then $\lambda$ is a dual Gamma function.
\end{prop}
\begin{proof} By Lemma \ref{wang5}, $(S,+)$ is a Clifford semigroup. By Lemma \ref{wangkangcccc}, $\lambda_x\in  \End(S,+)$ and $xy=x+\lambda_x(y)$ for all $x, y\in S$, and so the above $\lambda$ is indeed a well-defined function. Let $x,y\in S$. Then
$$\lambda_x\lambda_y(z)=\lambda_x(-y+yz)=\lambda_x(-y)+\lambda(yz)=-\lambda_x(y)+\lambda_x(yz)=-(-x+xy)-x+xyz$$$$=-(x-x+xy)+xyz=-(x+\lambda_x(y))+xyz
=-(x+\lambda_x(y))+(x+\lambda_x(y))z=\lambda_{x+\lambda_x(y)}(z)$$ and $-x+x+\lambda_x(y)=-x+x-x+xy=-x+xy=\lambda_x(y)$. Moreover, by Lemma \ref{wang5}, $$\lambda_{x^0}(x)=-x^0+x^0x=-x^0+x^0+x=x^0+x=x.$$ Thus (F1) holds. If $x^0=\lambda_x(x)=-x+xx$, then $$x=x+x^0=x-x+xx=x^0+xx=x^0xx=xx$$ by Lemma \ref{wang5}. This implies that $x\in E(S,\cdot)=E(S, +)$, and so $x=-x+x$. Thus (F2) holds.  (F3) follows from Lemma \ref{wangkangcccc} (2). If $e,f\in E(S, +)$, then $$\lambda_{e}(f)=-e+ef=e+ef=e+e+f=e+f$$ by Lemma \ref{wang5}. This gives (F4). Thus $\lambda$ is a Gamma function on $(S,+)$. The remaining assertion follows from Lemma \ref{jubutonggou}.
\end{proof}
\begin{prop}\label{good6}Let $(S,+)$ be a Clifford semigroup and $\gamma: (S,+)\rightarrow \End(S,+),\, x\mapsto \gamma_x$ be a Gamma function on $(S,+)$.
Define a binary operation $\circ$ on $S$ as follows: $$x\circ y=x+\gamma_x(y) \mbox{ for all } x,y\in S.$$ Then $(S, +, \circ)$ is a weak left brace and $x^{-1}$ appeared in Definition \ref{xx} is the inverse of $x$ in $(S, \circ)$ for all $x\in S$. In this case, we denote $(S, +, \circ)$ by ${\mathcal B}(\gamma)$. In particular, if $\gamma$ is a dual Gamma function, then $(S, +, \circ)$ is a dual weak left brace.
\end{prop}
\begin{proof}
Let $x,y,z\in S$. Then $$(x\circ y)\circ z=(x+\gamma_x(y))\circ z=x+\gamma_x(y)+\gamma_{x+\gamma_{x}(y)}(z)$$$$\overset{\rm (F1)}=x+\gamma_x(y)+\gamma_x\gamma_y(z) =x+\gamma_x(y+\gamma_y(z))=x\circ (y+\gamma_y(z))=x\circ(y\circ z).$$ This gives that $(S, \circ)$ is a semigroup. By (F3) and (F1), we have
$$x\circ x^{-1}=x+\gamma_{x}(x^{-1})=x-x=-x+x,$$
$$x\circ x^{-1}\circ x=(x-x)\circ x=(-x+x)\circ x$$$$=-x+x+\gamma_{-x+x}(x)=-x+x+x=x-x+x=x,$$
$$x^{-1}\circ x\circ x^{-1}=x^{-1}\circ (x-x)=x^{-1}+\gamma_{x^{-1}}(x-x)$$$$=x^{-1}+\gamma_{x^{-1}}(x)-\gamma_{x^{-1}}(x)=x^{-1}-x^{-1}+x^{-1}=x^{-1}.$$
This shows that $(S, \circ)$ is a regular semigroup and $E(S,+)\subseteq E(S,\circ)$. Conversely, Let $x\in E(S,\circ)$.  Then $x=x\circ x=x+\gamma_x(x)$, and hence $-x+x=-x+x+\gamma_x(x)=\gamma_x(x)$. This together with (F2) gives that $x=-x+x\in E(S,+)$. So $E(S,\circ)\subseteq E(S,+)$. Thus $E(S,+)= E(S,\circ)$. Let $e,f\in E(S,+)= E(S,\circ)$. Then $e\circ f=e+\gamma_e(f)=e+e+f=e+f$ by (F4). Dually, $f\circ e=f+e$. Therefore $e\circ f=f\circ e$. Thus $(S, \circ)$ forms an inverse semigroup and $x^{-1}$ is the inverse of $x$ in $(S, \circ)$.   Moreover,  $$x\circ (y+z)=x+\gamma_x(y+z)=x+\gamma_x(y)+\gamma_x(z)$$$$=x-x+x+\gamma_x(y)+\gamma_x(z)=x+\gamma_x(y)-x+x+\gamma_x(z)=x\circ y-x+x\circ z.$$
We have shown that  $(S, +, \circ)$ is a weak left brace. In particular, if $\gamma$ is a dual Gamma function, then $\gamma_x|_{H_x}$ is an automorphism of the subgroup $(H_x, +)$ of $(S,+)$. By the proof of Proposition \ref{good7}, we have $x^{-1}=\gamma^{-1}_x(-x)\in H_x=H_{x^{-1}}$ and  $$x^{-1}\circ x=x^{-1}+\gamma_{x^{-1}}(x)=x^{-1}-x^{-1}=\gamma^{-1}_x(-x)-\gamma^{-1}_x(-x)$$$$=\gamma^{-1}_x(-x)+\gamma^{-1}_x(x)=\gamma^{-1}_x(-x+x)
=\gamma^{-1}_x(x^0)=x^0=x\circ x^{-1}.$$ This shows that $(S,\circ)$ is a Clifford semigroup. So $(S, +, \circ)$ is a dual weak left brace.
\end{proof}
\begin{prop}\label{main2}
Let $(S,+)$ be a Clifford semigroup and denote
$$\mathbb{WB}^+=\{(S,+,\cdot)\mid (S,+,\cdot) \mbox{ forms a weak left brace}\},$$
$$\mathbb{GF}=\{\gamma \mid \gamma \mbox{ is a Gamma  function on }(S,+)\}$$
Then
$${\mathcal G}: \mathbb{WB}^+\rightarrow \mathbb{GF},\, (S,+,\cdot)\mapsto {\mathcal G}(S) \mbox{ and }{\mathcal B}:\mathbb{GF}\rightarrow\mathbb{WB}^+,\,\gamma\mapsto  {\mathcal B}(\gamma)$$ are mutually inverse bijections, and dual weak left braces correspond to dual Gamma functions.
\end{prop}
\begin{proof}By Propositions \ref{good5} and \ref{good6}, the above ${\mathcal G}$ and ${\mathcal B}$ are well-defined, and dual weak left braces correspond to dual Gamma functions..
Let $(S,+,\cdot)\in \mathbb{WB}^+$. By Proposition \ref{good5},  ${\mathcal G}(S)=\lambda$ and $\lambda_x(y)=-x+xy$  for all $x,y\in S$. Moreover, $x\circ y=x+\lambda_x(y)=xy$ by Proposition \ref{good6} and Lemma \ref{wangkangcccc}. Thus ${\mathcal {BG}}(S)={\rm id}_{\mathbb{WB}^+}$. Conversely, let $\gamma\in \mathbb{GF}$, ${\mathcal B}(\gamma)=(S, +, \circ)$ and ${\mathcal G}({\mathcal B}(\gamma))=\mu$. By Propositions \ref{good5} and \ref{good6}, we have $x\circ y=x+\gamma_x(y)$ and $\mu_x(y)=-x +x\circ y$  for all $x,y\in S$. In view of (F1), we obtain that $$\mu_x(y)=-x +x\circ y=-x+x+\gamma_x(y)=\gamma_x(y)$$  for all $x,y\in S$.
This yields that ${\mathcal G}({\mathcal B}(\gamma))=\gamma$. Thus ${\mathcal G}{\mathcal B}={\rm id}_{\mathbb{GF}}$.
\end{proof}
Combining Propositions \ref{good4} and \ref{main2}, we obtain the following result which is an analogue of \cite[Theorem 2.2]{Campedel} and \cite[Remark 2.3]{Caranti}.
\begin{theorem}\label{main1}Let $(S,+)$ be a Clifford semigroup. The  following data are equivalent:
\begin{itemize}
\item[(1)] A binary operation $\circ$ on $S$ such that $(S,+, \circ)$ is a  weak left brace (respectively, dual weak left brace).
\item[(2)] A good inverse (respectively,  Clifford) subsemigroup of $\End(S,+)\rtimes (S,+)$.
\item[(3)] A Gamma (respectively, dual Gamma) function on $(S,+)$.
\end{itemize}
\end{theorem}
Now we give an example to illustrate Theorem \ref{main1}.
\begin{exam}\label{example1}Consider   Clifford semigroup $S=\{0,e,f,a,b\}$ with $E(S,+)=\{0,e,f\}$ whose three subgroups are $\{0\}, \{e,a\}$ and $\{f,b\}$. The multiplication of $(S, +)$ is given in the first table in Table 1.
\begin{table}[H]
\begin{tabular}{r|rrrrr}
$+$ & $0$ & $e$ & $f$ & $a$&$b$\\
\hline
  $0$& $0$ & $0$ & $0$ & $0$&$0$\\
    $e$ &$0$ & $e$ & $0$ & $a$&$0$ \\
    $f$ &$0$ & $0$ & $f$ & $0$&$b$\\
    $a$ & $0$& $a$ & $0$ & $e$&$0$\\
    $b$ &$0$ & $0$ & $b$ & $0$&$f$
\end{tabular}
\hspace{3mm}
\begin{tabular}{r|rrrrr}
$\circ_1$ & $0$ & $e$ & $f$ & $a$&$b$\\
\hline
  $0$& $0$ & $0$ & $0$ & $0$&$0$\\
    $e$ &$0$ & $e$ & $0$ & $a$&$0$ \\
    $f$ &$0$ & $0$ & $f$ & $0$&$b$\\
    $a$ & $0$& $a$ & $0$ & $e$&$0$\\
    $b$ &$0$ & $0$ & $b$ & $0$&$f$
\end{tabular}
\hspace{3mm}
\begin{tabular}{r|rrrrr}
$\circ_2$ & $0$ & $e$ & $f$ & $a$&$b$\\
\hline
  $0$& $0$ & $0$ & $0$ & $0$&$0$\\
    $e$ &$0$ & $e$ & $0$ & $a$&$0$ \\
    $f$ &$0$ & $0$ & $f$ & $0$&$b$\\
    $a$ & $0$& $0$ & $a$ & $0$&$e$\\
    $b$ &$0$ & $b$ & $0$ & $f$&$0$
\end{tabular}
\vspace{3mm}
\caption{ }
\end{table}
 {\bf Step 1.} We first determine $\End(S,+)$. For convenience, we denote the map $$S\rightarrow S,\, 0\mapsto x_1, e\mapsto x_2, f\mapsto x_3, a\mapsto x_4, b\mapsto x_5$$ by $(x_1x_2x_3x_4x_5)$. Since $\phi(E(S, +))\subseteq E(S, +)$ for all $\phi\in \End(S,+)$, we can classify the endomorphisms of $(S,+)$ by $\phi(E(S, +))$.  By calculations, we have the following cases:

(1)  $|\phi(E(S, +))|=1$. In this case, $$\{\phi_1=(00000), \phi_2=(eeeee),\phi_3=(fffff)\}.$$

(2) $\phi(E(S, +))=\{0,e\}$. In this case,  $$\phi\in \{\phi_4=(00e0e), \phi_5=(00e0a), \phi_6=(0e0e0),\phi_7=(0e0a0)\}.$$

(3)  $\phi(E(S, +))=\{0,f\}$. In this case,   $$\phi\in \{\phi_8=(00f0f), \phi_9=(00f0b), \phi_{10}=(0f0f0),\phi_{11}=(0f0b0)\}.$$

(4) $\{e,f\}\subseteq \phi(E(S, +))$. \\

 {\bf Step 2.} Assume that $H$ is a good inverse subsemigroup of $\End(S,+)\rtimes (S,+)$. Then by (G1), for any $x\in S$, there is a unique $\phi\in H$ such that $(\phi, x)\in H$.

(1) If $(\phi, 0)\in H$, then by (G3) we have $\phi(y)=0+\phi(y)=0$ for all $y\in S$. So $\phi=\phi_1$.

(2) If $(\phi, e)\in H$ or $(\phi, a)\in H$, then $$\phi(y)=-a+a+\phi(y)=-e+e+\phi(y)=e+\phi(y),$$ which implies that
 $\Ima \phi\subseteq\{0,e,a\}$, and so
\begin{equation}\label{li1}
\phi\in \{\phi_1,\phi_2,\phi_4, \phi_5, \phi_6, \phi_7\}.
\end{equation}
  Let $(\phi, e)\in H$. Since $e,f\in E(S,+), e+f=0, e+a=a$, by Remark \ref{good3} we have $$(\phi, e)(\phi,e)=(\phi, e),\, (\phi,e)(\textunderscore\textunderscore, f)=(\phi_1, 0),\, (\phi, e)(\textunderscore\textunderscore, a)=(\textunderscore\textunderscore, a),$$ whence $e+\phi(e)=e, e+\phi(f)=0, e+\phi(a)=a$. Observe that $\phi(E(S,+))\subseteq E(S,+)$, it follows that $\phi(e)=e, \phi(f)\not=e, \phi(a)\not=e$. This together with (\ref{li1}) gives $\phi=\phi_7$, and so $(\phi_7,e)\in H$.

(3) Dually, if $(\phi, b)\in H$, then $\phi\in \{\phi_1,\phi_3,\phi_8, \phi_9, \phi_{10}, \phi_{11}\}$. Moreover, $(\phi_{11}, f)\in H$.

 {\bf Step 3.}  Observe that $E(H)=\{(\phi_1, 0), (\phi_7,e), (\phi_{11}, f)\}$ by   Remark \ref{good3}  and the fact that $E(S,+)=\{0,e,f\}$. Assume that the non-idempotent elements in $H$ are $(\phi,a)$ and $(\psi, b)$.  Then $(\phi,a)\not=(\phi,a)(\phi,a)=(\phi\phi, a+\phi(a))$. This together with (\ref{li1}) gives that $\phi\not=\phi_2, \phi_6$.  Dually, $\psi\not=\phi_3, \phi_{10}$.  Since $H$ is an inverse semigroup and the inverse of each idempotent is itself, we have the following two cases:

 {\em Case 1.} $(\phi,a)^{-1}=(\phi, a)$ and $(\psi, b)^{-1}=(\psi, b)$. In this case,
 $$(\phi\phi\phi, a+\phi(a)+\phi(\phi(a)))=(\phi,a)(\phi, a)(\phi, a)=(\phi, a).$$ This implies that
 $a+\phi(a)+\phi(\phi(a))=a$,  and so $\phi(a), \phi(\phi(a))\in \{e,a\}$.  This together with (\ref{li1}) gives $\phi\in \{\phi_2, \phi_6, \phi_7\}$.
But we already have known that $\phi\not=\phi_2, \phi_6$. So $\phi=\phi_7$. Dually, we have  $\psi=\phi_{11}$.  Thus we have $$H_1=\{(\phi_1,0), (\phi_7, e), (\phi_{11}, f), (\phi_7, a), (\phi_{11}, b)\}.$$ One can check that $H_1$ is indeed a good  Clifford subsemigroup of $\End(S,+)\rtimes (S,+)$.  In this case, ${\mathcal {B}}(H_1)=(S, \circ_1)$ and its  Cayley table is the second table in Table 1. In fact, $(S, \circ_1)=(S, +)$.  In this case, we obtain the trivial dual weak left brace $(S,+,+)$.

{\em Case 2.} $(\phi,a)$ and $(\psi, b)$ are mutually inverse. In this case, we have
$(\phi,a)(\psi, b)(\phi,a)=(\phi,a)$ and  $(\psi, b)(\phi,a)(\psi, b)=(\psi, b)$.  This implies that
\begin{equation}\label{li2}
\phi\psi\phi=\phi,\, a+\phi(b)+\phi(\psi(a))=a,\,\, \psi\psi\psi=\psi,\, b+\psi(b)+\psi(\phi(b))=b
\end{equation}
From these facts,  we obtain $\phi(b), \phi(\psi(a))\in \{e,a\}$ and  $\psi(a), \psi(\phi(b))\in \{f,b\}$. Since $\phi\not=\phi_2, \phi_6$, by (\ref{li1}) we have $\phi\in \{\phi_4, \phi_5\}$. Dually, we have $\psi\in \{\phi_8, \phi_9\}$. Using the facts (\ref{li2}), we obtain $\phi=\phi_4, \psi=\phi_8$ or $\phi=\phi_5, \psi=\phi_9$.  If $\phi=\phi_4, \psi=\phi_8$, then $(\phi_4, a), (\phi_8, b)\in H$, and so $(\phi_6, a)=(\phi_4\phi_8, a+\phi_4(b))=(\phi_4,a)(\phi_8,b)\in H$.
This yields $(\phi_4, a),(\phi_6, a)\in H$.  By (G1) this is impossible. Thus  $\phi=\phi_5, \psi=\phi_9$, and we have $$H_2=\{(\phi_1,0), (\phi_7, e), (\phi_{11}, f), (\phi_5, a), (\phi_{9}, b)\}.$$ One can check  $H_2$ is a good inverse (not Clifford!) subsemigroup of $\End(S,+)\rtimes (S,+)$.  In this case, ${\mathcal {B}}(H_2)=(S, \circ_2)$ and its Cayley table is the third table in Table 1. In fact, $(S, \circ_2)$ is  the Brandt semigroup $B_2$ (see \cite{A10}, page 32).  In this case, we obtain the  (not dual!) weak left brace $(S,+,\circ_2)$.

 {\bf Step 4.}  In conclusion, for the given Clifford semigroup $(S, +)$, we have
$$\mathbb{WB}^+=\{(S,+,\cdot)\mid (S,+,\cdot) \mbox{ forms a weak left brace}\}=\{(S,+,\circ_1), (S,+,\circ_2)\},$$
$$\mathbb{GIS}=\{H\mid H \mbox{ is a good inverse subsemigroup of }\End(S,+)\rtimes (S,+)\}=\{H_1, H_2\}$$
and $(S,+,\circ_i)$ corresponds to $H_i$, $i=1,2$.  Moreover, denote
$$\gamma_1: (S,+)\rightarrow \End(S,+),\,\, 0\mapsto \phi_1, e\mapsto \phi_7, f\mapsto \phi_{11}, a\mapsto \phi_7, b\mapsto \phi_{11},$$
$$\gamma_2: (S,+)\rightarrow \End(S,+),\,\, 0\mapsto \phi_1, e\mapsto \phi_7, f\mapsto \phi_{11}, a\mapsto \phi_5, b\mapsto \phi_{9}.$$
Then we have $$\mathbb{GF}=\{\gamma \mid \gamma \mbox{ is a Gamma  function on }(S,+)\}=\{\gamma_1, \gamma_2\}$$ and $(S,+,\circ_i)$ corresponds to $\gamma_i$, $i=1,2$, where  $\gamma_1$ is a dual Gamma function.
\end{exam}

Finally, we characterize weak left braces by so-called affine structures on inverse semigroups. Observe that left braces and  left semi-braces have been characterized by affine and semi-affine structures on groups in \cite{Rump3} and \cite{Stefanelli1}, respectively.   We begin by giving the following notion.
\begin{defn}\label{affine0}
Let $(S, \cdot)$ be an inverse semigroup and $\diamond$ be a binary operation on $S$.   Then $\diamond$ is called an affine structure on $(S, \cdot)$ if the following conditions hold: For all $a\in S$ and $e\in E(S, \cdot)$,
\begin{itemize}
\item[(\rm {A1})] $(ab)\diamond c=b\diamond(a\diamond c)$.
\item[(\rm A2)]$a\diamond (b(b\diamond c))=(a\diamond b)((a\diamond b)\diamond (a\diamond c))$.
\item[(\rm A3)] $e\diamond a=ea,\,\, a\diamond e=a^{-1}ea$.
\end{itemize}
\end{defn}
\begin{lemma}\label{affine1}Let $(S, \cdot)$ be an inverse semigroup and $\diamond$ be a binary operation on $S$. Define $x+y=x(x\diamond y)$ for all $x,y\in S$.
\begin{itemize}
\item[(1)] If (A1) and (A3) hold, then for all $(S, \cdot)$, $a,b\in S$ and $e\in E(S, \cdot)$, we have $$(aa^{-1})\diamond a=a,\,\, (aa^{-1})(a^{-1}\diamond b)=a^{-1}\diamond b,\,\, a\diamond b=a^{-1}(a+b),\,\,    e+a=a+e=ea.$$
\item[(2)] If (A1) and (A3) hold,  (A2) is equivalent to the axiom $(a+b)+c=a+(b+c)$.
\item[(3)] The axiom (A2) is exactly the axiom $a\diamond(b+c)=(a\diamond b)+(a\diamond c)$.
\end{itemize}

\end{lemma}
\begin{proof}(1) Firstly,
$(aa^{-1})\diamond a=aa^{-1}a=a$ by (A3). Secondly, by (A3) and (A1), we have  $$(aa^{-1})(a^{-1}\diamond b)=(aa^{-1})\diamond (a^{-1}\diamond b)=(a^{-1}aa^{-1}) (\diamond b)=a^{-1}\diamond b.$$ This implies that $a^{-1}(a+b)=a^{-1}a(a+b)=a+b$.
Finally, we have $e+a=e(e\diamond a)=eea=ea$ and $a+e=a(a\diamond e)=aa^{-1}ea=eaa^{-1}a=ea$ by  (A3) and the fact that  $(S, \cdot)$ is an inverse semigroup.

(2) Assume that (A1) and (A3) hold.  Observe that $a+(b+c)=a(a\diamond (b(b\diamond c)))$ $$(a+b)+c=a(a\diamond b)+c=a(a\diamond b)((a(a\diamond b))\diamond c)=a(a\diamond b)((a\diamond b)\diamond (a\diamond c))$$ by (A1). If (A2) holds, them we have $a+(b+c)=(a+b)+c$ certainly.  Conversely, if  $a+(b+c)=(a+b)+c$, then by the second equation in item (1) of the present lemma,  $$(a\diamond b)((a(a\diamond b))=a^{-1}a(a\diamond b)((a(a\diamond b))=a^{-1}(a+(b+c))$$$$=a^{-1}((a+b)+c)=a^{-1}a(a\diamond b)((a\diamond b)\diamond (a\diamond c))=(a\diamond b)((a\diamond b)\diamond (a\diamond c)).$$

(3) This is obvious.
\end{proof}
\begin{prop}\label{affine2}
Let $(S, \cdot)$ be an inverse semigroup with an affine structure $\diamond$. Define $x+y=x(x\diamond y)$ for all $x,y\in S$. Then $(S, +, \cdot)$ forms a weak left brace and we denote ${\mathcal{B}}(\diamond)=(S, +, \cdot)$. Moreover,
\begin{itemize}
\item[(1)]$(S, +, \cdot)$ is a dual weak left brace if and only if $(S,\cdot)$ is a Clifford semigroup.
\item[(2)]$(S, +, \cdot)$ is a skew left brace if and only if $(S,\cdot)$ is a group. In this case, (A3) reduces to the axioms $1_{\rm S}\diamond a=a,\,\, a\diamond 1_{\rm S}=1_{\rm S}$.
\end{itemize}

\end{prop}
\begin{proof}
By Lemma \ref{affine1} (2), $(S,+)$ forms a semigroup. Let $a\in S$. Then
\begin{equation}\label{affine3}
a+(a^{-1}\diamond a^{-1})=a(a\diamond (a^{-1}\diamond a^{-1}))\overset{\rm (A1)}{=}a((a^{-1}a)\diamond a^{-1})\overset{\rm (A3)}{=}aa^{-1}aa^{-1}=aa^{-1}. \end{equation}
This implies that
\begin{equation}\label{affine4}
a+(a^{-1}\diamond a^{-1})+a=aa^{-1}+a=aa^{-1}a=a
\end{equation} by Lemma \ref{affine1} (1). Thus $(S, +)$ is a regular semigroup.

We next assert that $E(S,\cdot)=E(S, +)$. If $e\in E(S,+)$, then $e+e=ee=e$  by Lemma \ref{affine1} (1), and so $e\in E(S, +)$.  Conversely, if $a\in E(S, +)$, then $a=a+a=a(a\diamond a)$, and so $a^{-1}a=a^{-1}a(a\diamond a)=a\diamond a$ by Lemma \ref{affine1} (1). This implies that
$$aa^{-1}=(a^{-1})^{-1}a^{-1}aa^{-1}=\overset{\rm (A3)}{=}a^{-1}\diamond(a^{-1}a)=a^{-1}\diamond(a\diamond a)$$$$\overset{\rm (A1)}{=}(aa^{-1})\diamond a\overset{\rm Lemma\, \ref{affine1}\, (1)}{=}aa^{-1}a=a,$$ and so $a\in E(S, +)$.  Thus $E(S,\cdot)=E(S, +)$.

Now let $a\in S$ and $e\in E(S,\cdot)=E(S, +)$. Then $e+a=ea=a+e$ by Lemma \ref{affine1} (1). Therefore $(S,+)$ is a  Clifford semigroup by Lemma \ref{Clifford}, and so $(S,+)$ is an inverse semigroup.
Moreover, we have $$(a^{-1}\diamond a^{-1})+a+(a^{-1}\diamond a^{-1})\overset{\rm (\ref{affine3})}{=}(a^{-1}\diamond a^{-1})+aa^{-1}$$$$\overset{\rm Lemma\, \ref{affine1}\, (1)}{=}aa^{-1}(a^{-1}\diamond a^{-1})\overset{\rm Lemma\, \ref{affine1}\, (1)}{=}a^{-1}\diamond a^{-1}$$
This together with (\ref{affine4}) gives that $-a=a^{-1}\diamond a^{-1}$. Furthermore,
\begin{equation}\label{affine5}
-a+a=a-a=a+(a^{-1}\diamond a^{-1})=aa^{-1}
\end{equation}
by Lemma \ref{jichu} and (\ref{affine3}).

Let $a,b,c\in S$. Observe that $a(b+c)=ab(b\diamond c)$ and
\begin{equation}\label{affine6}
\begin{aligned}
&(ab)+a(a^{-1}+c)=ab((ab)\diamond (aa^{-1}(a^{-1}\diamond c)))\overset{\rm Lemma\, \ref{affine1}\, (1)}{=}ab((ab)\diamond (a^{-1}\diamond c))\\
&\overset{\rm (A1)}{=}
ab((a^{-1}ab)\diamond c)\overset{\rm (A1)}{=}ab(b\diamond (a^{-1}a\diamond c))\overset{\rm (A3)}{=}ab(b\diamond (a^{-1}a c)).
\end{aligned}
\end{equation}
This gives that
\begin{equation}\label{affine7}
\begin{aligned}
&b^{-1}a^{-1}ab(b\diamond c)\overset{\rm (A3)}{=}(b^{-1}a^{-1}ab)\diamond(b\diamond c)\overset{\rm (A1)}{=}(bb^{-1}a^{-1}ab)\diamond c\\
&=(a^{-1}abb^{-1}b)\diamond c=(a^{-1}ab)\diamond c\overset{\rm (A1)}{=}b\diamond (a^{-1}a\diamond c)\overset{\rm (A3)}{=}b\diamond (a^{-1}a c).
\end{aligned}
\end{equation}
Substituting $c$ by $aa^{-1}c$ in (\ref{affine7}), we have
\begin{equation}\label{affine8}
b^{-1}a^{-1}ab(b\diamond a^{-1}ac)=b\diamond (a^{-1}a a^{-1}ac)=b\diamond (a^{-1}a c)=b^{-1}a^{-1}ab(b\diamond c).
\end{equation}
Multiplying $ab$ from the left on both sides in (\ref{affine8}), we have
\begin{equation}\label{affine9}
a(b+c)= ab b^{-1}a^{-1}ab(b\diamond c)=abb^{-1}a^{-1}ab(b\diamond a^{-1}ac)\overset{\rm (\ref{affine6})}=ab+a(a^{-1}+c).
\end{equation}
Substituting $b$ by
$a^{-1}a$ in (\ref{affine6}), we have
\begin{equation*}
\begin{aligned}
a+a(a^{-1}+c)=aa^{-1}a+a(a^{-1}+c)
=aa^{-1}a (a^{-1}a\diamond (a^{-1}a c))\overset{\rm (A3)}{=}
aa^{-1}a a^{-1}aa^{-1}a c=ac,
\end{aligned}
\end{equation*}
$$-a+ac=-a+a+a(a^{-1}+c)\overset{\rm(\ref{affine5})}{=}aa^{-1}+a(a^{-1}+c)\overset{\rm Lemma\, \ref{affine1}\, (1)}{=}aa^{-1}a(a^{-1}+c)=a(a^{-1}+c)$$
This together with (\ref{affine5}) and (\ref{affine9}) gives that $a(b+c)=ab-a+ac$ and $-a+a=aa^{-1}$.
Since $(S,+)$ is an inverse semigroup, we obtain that $(S, +, \cdot)$ forms a weak left brace.  The remaining part is obvious.
\end{proof}
\begin{prop}\label{affine10}
Let $(S, +, \cdot)$ be a dual weak left brace. Define $$a \diamond b=\lambda_{a^{-1}}(b)=-a^{-1}+a^{-1}b=a^{-1}(a+b) \mbox { for all } a,b\in S.$$  Then $\diamond$ is an
affine stricture on  the inverse semigroup $(S, \cdot)$ and we denote $\diamond={\mathcal A}(S)$.
\end{prop}
\begin{proof}We first prove (A1) and (A3). Let $a,b,c\in S$ and $e\in E(S, \cdot)=E(S, +)$. Then by Lemma \ref{wangkangcccc}, $\lambda$ is a homomorphism, and so    $$(ab)\diamond c=\lambda_{(ab)^{-1}}(c)=\lambda_{b^{-1}a^{-1}}(c)=\lambda_{b^{-1}}\lambda_{a^{-1}}(c)=\lambda_{b^{-1}}(a\diamond c)=b\diamond (a\diamond c).$$
By Lemma \ref{wang5} we have $e\diamond a=e^{-1}(e+a)=e(ea)=ea$ and $a\diamond e=a^{-1}(a+e)=a^{-1}ea$.
Observe that
\begin{equation}\label{wangg}
x+'y=x(x\diamond y)=xx^{-1}(x+y)=xx^{-1}+x+y=-x+x+x+y=x-x+x+y=x+y
\end{equation}
 by Lemma \ref{wang5} and (\ref{dengshi}), (A2) follows from Lemma \ref{affine1} (2) and the fact that $(S,+)$ is an inverse semigroup by hypothesis.
\end{proof}
The following result is an analogue of \cite[Theorem 2.1]{Rump3} and \cite[Corollary 12]{Stefanelli1}.
\begin{theorem}\label{main11}Let $(S,\cdot)$ be an inverse semigroup. Denote
$$\mathbb{WB}^\circ=\{(S,+,\cdot)\mid (S,+,\cdot) \mbox{ forms a weak left brace}\},$$
$$\mathbb{AS}=\{\diamond \mid \diamond \mbox{ is an affine  structure on }(S,\cdot)\}$$
Then
${\mathcal A}: \mathbb{WB}^\circ \rightarrow \mathbb{AS},\, (S,+,\cdot)\mapsto {\mathcal A}(S) \mbox{ and }{\mathcal B}:\mathbb{AS}\rightarrow\mathbb{WB}^\circ,\,\gamma\mapsto  {\mathcal B}(\diamond)$ are mutually inverse bijections.
\end{theorem}
\begin{proof}
By Propositions \ref{affine2} and  \ref{affine10}, the above ${\mathcal A}$ and ${\mathcal B}$ are well-defined.  Assume that $(S,+,\cdot)\in \mathbb{WB}^\circ$, ${\mathcal{A}}(S)=\diamond$ and ${\mathcal{BA}}(S)=(S, +', \cdot)$. Then by (\ref{wangg}), we have $+=+'$, and so ${\mathcal{BA}}={\rm id}_{\mathbb{WB}^\circ}$.
Conversely, assume that $\diamond\in \mathbb{AS}, {\mathcal B}(\diamond)=(S, +,\cdot)$ and ${\mathcal AB}(\diamond)=\diamond'$. Then for all $a,b\in S$, we have
$a\diamond' b=a^{-1}(a+b)=a^{-1}a(a\diamond b)=a\diamond b$ by  Propositions \ref{affine2} and  \ref{affine10} and Lemma \ref{affine1} (1).  This implies that  $\diamond=\diamond'$,  and so
${\mathcal AB}={\rm id}_{\mathbb{AS}}$. Thus, the desired results  hold.
\end{proof}
Theorem \ref{main11} can be illustrated by the following example.
\begin{exam}Consider the Brandt semigroup $S=\{0,e,f,a,b\}$ (see \cite{A10}, page 32). This is an inverse semigroup whose multiplication table the first one in Table 2.
\begin{table}[H]
\begin{tabular}{r|rrrrr}
$\cdot$ & $0$ & $e$ & $f$ & $a$&$b$\\
\hline
  $0$& $0$ & $0$ & $0$ & $0$&$0$\\
    $e$ &$0$ & $e$ & $0$ & $a$&$0$ \\
    $f$ &$0$ & $0$ & $f$ & $0$&$b$\\
    $a$ & $0$& $0$ & $a$ & $0$&$e$\\
    $b$ &$0$ & $b$ & $0$ & $f$&$0$
\end{tabular}
\hspace{2mm}
\begin{tabular}{r|rrrrr}
$\diamond$ & $0$ & $e$ & $f$ & $a$&$b$\\
\hline
  $0$& $0$ & $0$ & $0$ & $0$&$0$\\
    $e$ &$0$ & $e$ & $0$ & $a$&$0$ \\
    $f$ &$0$ & $0$ & $f$ & $0$&$b$\\
    $a$ & $0$& $f$ & $0$ & $?$&$?$\\
    $b$ &$0$ & $0$ & $e$ & $?$&$?$
\end{tabular}
\hspace{2mm}
\begin{tabular}{r|rrrrr}
$+$ & $0$ & $e$ & $f$ & $a$&$b$\\
\hline
  $0$& $0$ & $0$ & $0$ & $0$&$0$\\
    $e$ &$0$ & $e$ & $0$ & $a$&$0$ \\
    $f$ &$0$ & $0$ & $f$ & $0$&$b$\\
    $a$ & $0$& $a$ & $0$ & $?$&$?$\\
    $b$ &$0$ & $0$ & $b$ & $?$&$?$
\end{tabular}
\vspace{2mm}
 \caption{}
\end{table}
Then we have $a^{-1}=b, b^{-1}=a, 0^{-1}=0, e^{-1}=e, f^{-1}=f$ and $E(S,\cdot)=\{0,e,f\}$. We shall try to determine all affine structures on $(S, +)$.
According to Proposition \ref{affine2}, we have the second and the third tables in Table 2. By Lemma \ref{affine1} (3), we have the axiom $x\diamond(y+z)=(x\diamond y)+(x\diamond z)$. Moreover, by the proof of Proposition \ref{affine2}, $(S,+)$ is a Clifford semigroup and $E(S,+)=\{0,e,f\}$.
Since $a\diamond(f+b)=(a\diamond f)+(a\diamond b)=0+(a\diamond b)=0$, we have $a+b=a(a\diamond b)=0$.  Dually, we have $b\diamond a=0$ and $b+a=0$.
Observe that  $a\diamond a=a\diamond(e+a)=(a\diamond e)+(a\diamond a)=f+(a\diamond a)$. This implies that $a\diamond a\in \{0,f, a\}$ by Table 1. If $a\diamond a=0$, then $a+a=a(a\diamond a)=0$, and so $a=a-a+a=-a+a+a=-a+0=0$ by Lemma \ref{jichu}, a contradiction. If
$a\diamond a=f$,  then $a+a=a(a\diamond a)=af=a$, and so $a\in E(S, +)$, a contradiction. If
$a\diamond a=a$,  then $a+a=a(a\diamond a)=aa=e$. Dually, we have $b\diamond b=b$ and $b+b=f$. In this case, we have
\begin{table}[H]
\begin{tabular}{r|rrrrr}
$\cdot$ & $0$ & $e$ & $f$ & $a$&$b$\\
\hline
  $0$& $0$ & $0$ & $0$ & $0$&$0$\\
    $e$ &$0$ & $e$ & $0$ & $a$&$0$ \\
    $f$ &$0$ & $0$ & $f$ & $0$&$b$\\
    $a$ & $0$& $0$ & $a$ & $0$&$e$\\
    $b$ &$0$ & $b$ & $0$ & $f$&$0$
\end{tabular}
\hspace{2mm}
\begin{tabular}{r|rrrrr}
$\diamond_1$ & $0$ & $e$ & $f$ & $a$&$b$\\
\hline
  $0$& $0$ & $0$ & $0$ & $0$&$0$\\
    $e$ &$0$ & $e$ & $0$ & $a$&$0$ \\
    $f$ &$0$ & $0$ & $f$ & $0$&$b$\\
    $a$ & $0$& $f$ & $0$ & $a$&$0$\\
    $b$ &$0$ & $0$ & $e$ & $0$&$b$
\end{tabular}
\hspace{2mm}
\begin{tabular}{r|rrrrr}
$+_1$ & $0$ & $e$ & $f$ & $a$&$b$\\
\hline
  $0$& $0$ & $0$ & $0$ & $0$&$0$\\
    $e$ &$0$ & $e$ & $0$ & $a$&$0$ \\
    $f$ &$0$ & $0$ & $f$ & $0$&$b$\\
    $a$ & $0$& $a$ & $0$ & $e$&$0$\\
    $b$ &$0$ & $0$ & $b$ & $0$&$f$
\end{tabular}
\vspace{2mm}
 \caption{}
\end{table}
\noindent One can routinely check that (A1), (A2) and (A3) hold for $\diamond_1$. So $\diamond_1$ is an affine structure on $(S,\cdot)$  and we obtain the unique weak left brace $(S, +_1, \cdot)$ whose multiplication is the given Brandt semigroup $(S, \cdot)$, where $+_1$ is given in the third table in Table 3.

In conclusion, for the given Brandt semigroup $(S, \cdot)$, we have
$$\mathbb{WB}^\circ=\{(S,+,\cdot)\mid (S,+,\cdot) \mbox{ forms a weak left brace}\}=\{(S,+_1,\cdot)\},$$
$$\mathbb{AS}=\{\diamond \mid \diamond \mbox{ is an affine  structure on }(S,\cdot)\}=\{\diamond_1\}.$$
\end{exam}

\section{Symmetric, $\lambda$-homomorphic and $\lambda$-anti-homomorphic weak left braces}
Inspired by symmetric, $\lambda$-homomorphic and $\lambda$-anti-homomorphic skew left braces appeared in literature, in this section we introduce symmetric, $\lambda$-homomorphic and $\lambda$-anti-homomorphic weak left braces and consider their characterizations and  constructions.

Symmetric  skew left braces were firstly introduced in the paper of Childs  \cite{Childs} under the name bi-skew brace and studied by Caranti \cite{Caranti}, and the name of ``symmetric  skew left braces" was firstly formulated in \cite{Bardakov1}.
The following notion can be regarded as a generalization of symmetric  skew left braces.
\begin{defn}
Let $(S, +)$ and $(S, \cdot)$ be two inverse semigroups. Then $(S, +, \cdot)$  is called  a {\em symmetric weak left brace} if the following axioms hold:
\begin{equation}\label{symmetric}
\begin{aligned}
x(y+z)=xy-x+xz,\\
x+(yz)=(x+y)x^{-1}(x+z).
\end{aligned}
\end{equation}
From \cite{Bardakov1,Childs}, a symmetric weak left brace $(S, +, \cdot)$ is called {\em a symmetric  skew left brace} if both $(S, +)$ and $(S, \cdot)$  are groups.
\end{defn}
We first point out that symmetric weak left braces are special dual weak left braces.
\begin{prop}\label{duichen}
Let $(S, +, \cdot)$  be a symmetric weak left brace. Then $-x+x=xx^{-1}$ and $x-x=x^{-1}x$ for all $x\in S$.  As a consequence, $(S, +, \cdot)$  is a dual weak left brace.
\end{prop}
\begin{proof} We first prove that $E(S,+)=E(S, \cdot)$. Let $e\in E(S,+)$. Then $e=-e=e+e$. By (\ref{symmetric}), we have
$$ee=e(e+e)=ee-e+ee=ee+e+ee,\,\, e+ee=(e+e)e^{-1}(e+e)=ee^{-1}e=e,$$  whence
$ee+e+ee=ee$ and $e+ee+e=e+e=e$. Since $(S, +)$ is an inverse semigroup, we have $ee=-e=e\in E(S, \cdot)$. Thus $E(S,+)\subseteq E(S, \cdot)$. Dually,  $E(S, \cdot)\subseteq E(S, +)$.

Next, we  show that $ef=e+f$ for all $e,f\in E(S,+)= E(S, \cdot)$. In fact, by (\ref{symmetric}) and the fact that $(S, +)$ and $(S, \cdot)$  are inverse semigroups, by Lemma \ref{inverse} we have
$$ef=e(f+f)=ef-e+ef=ef+e+ef=e+ef,$$
$$e+f=e+(ff)=(e+f)e^{-1}(e+f)=(e+f)e(e+f)$$$$=e(e+f)=ee-e+ef=ee+e+ef=e+ef.$$
This gives that $ef=e+f$.

Finally, we prove $-x+x=xx^{-1}$ and $x-x=x^{-1}x$ for all $x\in S$.  We only prove the first equation, and the other equation  can be proved dually. Let $x,y\in S$.
Then $x(y+x^{-1}x)=xy-x+xx^{-1}x=xy-x+x$. On the other hand,
$$x(y+x^{-1}x)=x(y-y+y+x^{-1}x)=x(y+x^{-1}x-y+y)$$$$=xy-x+x(x^{-1}x-y+y)=xy-x+xx^{-1}x-x+x(-y+y)$$$$=xy-x+x-x+x(-y+y)=xy-x+x(-y+y)=x(y-y+y)=xy.$$
This gives that
\begin{equation}\label{symmetric1} xy=xy-x+x \mbox{ for all } x,y\in S.
\end{equation}
Moreover, $$-x+x(-x^{-1}+x^{-1})=-x+xx^{-1}x-x+x(-x^{-1}+x^{-1})$$$$=-x+x(x^{-1}x-x^{-1}+x^{-1})\overset{\rm(\ref{symmetric1})}{=} -x+xx^{-1}x=-x+x,$$
$$-x(-x^{-1}+x^{-1})+x=-(-x+x(-x^{-1}+x^{-1}))=-(-x+x)=-x+x.$$
On the other hand, by (\ref{symmetric1}) we have $$x=(xx^{-1})x=(xx^{-1})x-xx^{-1}+xx^{-1}=x+xx^{-1}+xx^{-1}=x+xx^{-1}.$$
Again by (\ref{symmetric1}), we obtain
\begin{equation*}
\begin{aligned}
&x(-x^{-1}+x^{-1})=x(-x^{-1}+x^{-1})-x+x\\
&=x(-x^{-1}+x^{-1})-x(-x^{-1}+x^{-1})+x\\
&=x(-x^{-1}+x^{-1})-x(-x^{-1}+x^{-1})+x-x+x\\
&=x-x+x(-x^{-1}+x^{-1})-x(-x^{-1}+x^{-1})+x\\
&=x-x+x(-x^{-1}+x^{-1})
=x-x+x=x.
\end{aligned}
\end{equation*}
Substituting $x$ by $x^{-1}$ in the above equation, we have $x^{-1}(-x+x)=x^{-1}$, and so $$x=(x^{-1})^{-1}=(x^{-1}(-x+x))^{-1}=(-x+x)^{-1}x=(-x+x)x.$$
This implies that $$-x+x=-x+(x+xx^{-1})=(-x+x)+(xx^{-1})=(-x+x)(xx^{-1})=xx^{-1},$$ as required. By the above statements, both $(S,+)$ and $(S,\cdot)$ are Clifford semigroups, and hence $(S, +, \cdot)$  is a dual weak left brace.
\end{proof}
In 2022, $\lambda$-homomorphic skew left braces are introduced in \cite{Bardakov1}.  Inspired this notion, we can define $\lambda$-homomorphic weak left braces.
\begin{defn}
A weak left brace $(S, +, \cdot)$ is called {\em $\lambda$-homomorphic} if the map $\lambda$ given in Lemma \ref{wangkangcccc} is a homomorphism from $(S, +)$ to $\End (S,+)$, that is,  for all $x,y,z\in S$,
\begin{equation}\label{lambda-homomorphic}
\lambda_x\lambda_y(z)=-x+x(-y+yz)=-y-x+(x+y)z=\lambda_{x+y}(z)
\end{equation}
From \cite{Bardakov1,Childs}, a $\lambda$-homomorphic weak left brace $(S, +, \cdot)$ is called {\em a $\lambda$-homomorphic  skew left brace} if $(S, +, \cdot)$ is a skew left brace.
\end{defn}
\begin{prop}\label{tongtai}
A $\lambda$-homomorphic  weak left brace $(S, +, \cdot)$ is  necessarily a  dual weak left brace.
\end{prop}
\begin{proof}Let $x,y,z\in S$.  On one hand,
\begin{equation*}
\begin{aligned}
&x(y+x+(-x)(-x+xz)) =xy-x+x(x+(-x)(-x+xz))\,\,\,\,\,    (\mbox{by (\ref{dengshi})})\\
&=xy+x-x+(x-x)(-x+xz)\,\,\,\,\,    (\mbox{by  (\ref{lambda-homomorphic})})\\
&=xy+x-x+x-x-x+xz\,\,\,\,\,    (\mbox{by  Lemma \ref{wang5} and the fact } x-x\in E(S))\\
&=xy-x+xz=x(y+z).\,\,\,\,\,    (\mbox{by (\ref{dengshi})})
\end{aligned}
\end{equation*}
On the other hand,
\begin{equation*}
\begin{aligned}
&x+(-x)(-x+xz)=-x+x+(-x+x)z\,\,\,\,\,    (\mbox{by  (\ref{lambda-homomorphic})})\\
&=-x+x-x+x+z=-x+x+z,\,\,\,\,\,    (\mbox{by  Lemma \ref{wang5}})
\end{aligned}
\end{equation*}
\begin{equation*}
\begin{aligned}
&x(y+x+(-x)(-x+xz)) =x(y+(-x+x)+z)=x((-x+x)+y+z)\\
&=x(-x+x)(y+z). \,\,\,\,\,    (\mbox{by  Lemma \ref{wang5} and the fact } x-x\in E(S))
\end{aligned}
\end{equation*}Thus we have $x(-x+x)(y+z)=x(y+z)$ for all $x,y,z$ in $S$.  This implies that
$$x^{-1}(-x^{-1}+x^{-1})(-x+x)=x^{-1}(-x+x)=x^{-1}xx^{-1}=x^{-1}.$$
On the other hand, $$x^{-1}(-x^{-1}+x^{-1})(-x+x)=x^{-1}(-x+x)(-x^{-1}+x^{-1})=x^{-1}xx^{-1}x^{-1}x=x^{-1}x^{-1}x.$$
So $x^{-1}=x^{-1}x^{-1}x$, whence $xx^{-1}=xx^{-1}x^{-1}x$. Substituting $x$ by $x^{-1}$ in this equation, we obtain $x^{-1}x=x^{-1}xxx^{-1}$.
This gives $xx^{-1}=xx^{-1}x^{-1}x=x^{-1}xxx^{-1}=x^{-1}x$. Hence $(S,\cdot)$ is also a Clifford semigroup and $(S,+, \cdot)$ is a dual weak left brace.
\end{proof}

Recently, $\lambda$-anti-homomorphic skew left braces are introduced in \cite{Bardakov2}. We can extend this notion as follows.

\begin{defn}
A weak left brace $(S, +, \cdot)$ is called {\em $\lambda$-anti-homomorphic} if the map $\lambda$ given in Lemma \ref{wangkangcccc} is an $\lambda$-anti-homomorphism from $(S, +)$ to $\End (S,+)$, that is,  for all $x,y,z\in S$,
\begin{equation}\label{lambda-antihomomorphic1}
\lambda_y\lambda_x(z)=-y+y(-x+xz)=-y-x+(x+y)z=\lambda_{x+y}(z).
\end{equation}
From \cite[Definition 3.1]{Bardakov2}, a $\lambda$-anti-homomorphic weak left brace $(S, +, \cdot)$ is called a {\em $\lambda$-anti-homomorphic skew left brace} if $(S, +, \cdot)$ is a skew left brace.
\end{defn}
\begin{lemma}A weak left brace $(S, +, \cdot)$ is $\lambda$-anti-homomorphic if and only if the following axiom holds:
\begin{equation}\label{lambda-antihomomorphic2}
 y(-x+xz)=-x+(x+y)z.
\end{equation}
\end{lemma}
\begin{proof}Obviously, (\ref{lambda-antihomomorphic2}) implies  (\ref{lambda-antihomomorphic1}). To show the converse, we first observe that
$$y+(-y+y(-x+xz))=yy^{-1}+y(-x+xz)=yy^{-1}y(-x+xz)=y(-x+xz)$$ by Lemma \ref{wang5}.
Moreover, by Lemma \ref{wang5} again, we have $$y+(-y-x+(x+y)z)=yy^{-1}-x+(x+y)z=-x+yy^{-1}+(x+y)z
=-x+yy^{-1}(x+y)z$$$$=-x+(yy^{-1}+(x+y))z=-x+(x+yy^{-1}+y)z=-x+(x+yy^{-1}y)z=-x+(x+y)z.$$
Thus (\ref{lambda-antihomomorphic1}) also implies  (\ref{lambda-antihomomorphic2}).
\end{proof}
\begin{prop}\label{fantongtai}
A $\lambda$-anti-homomorphic  weak left brace $(S, +, \cdot)$ is  necessarily a dual weak left brace.
\end{prop}
\begin{proof}Let $x,y,z\in S$. We first observe by (\ref{dengshi}) and Lemma \ref{wang5} that $$-x^{-1}+x^{-1}(-x)=x^{-1}x-x^{-1}+x^{-1}(-x)=x^{-1}(x-x)=x^{-1}xx^{-1}=x^{-1},$$ $-x+x(-x^{-1})=x$ and  $$x^{-1}(-x+xx)=x^{-1}(-x)-x^{-1}+x^{-1}xx=x^{-1}(-x)-x^{-1}+x^{-1}x+x$$$$=x^{-1}x+x^{-1}(-x)-x^{-1}+x=x^{-1}-x^{-1}+x^{-1}(-x)-x^{-1}+x
=x^{-1}+x^{-1}-x^{-1}+x=x^{-1}+x.$$This implies that
$$-x+x^{-1}(x+(-x+xx)x)=-x+x^{-1}x-x^{-1}+x^{-1}(-x+xx)x=-x-x^{-1}+x^{-1}(-x+xx)x$$$$=-x-x^{-1}+(x^{-1}+x)x\overset{(\ref{lambda-antihomomorphic1})}{=}-x+x(-x^{-1}+x^{-1}x)
=-x+x(-x^{-1})=x.$$
On the other hand, by Lemma \ref{wang5} we have
\begin{equation}\label{lambda-antihomomorphic3}
\begin{aligned}
&-x+y(x+(-x)z)\overset{(\ref{lambda-antihomomorphic2})}{=}-x+x+(-x+y)z\\
&=(-x+x)(-x+y)z=(-x+x-x+y)z=(-x+y)z.
\end{aligned}
\end{equation}
Take $y=z=x$ in (\ref{lambda-antihomomorphic3}). Then
\begin{equation}\label{lambda-antihomomorphic4}
 -x+x(x+(-x)x)=(-x+x)x=xx^{-1}x=x.
\end{equation}
Adding $x$ from the left on the both sides of (\ref{lambda-antihomomorphic3}), we have
\begin{equation}\label{lambda-antihomomorphic5}
 (x-x)y(x+(-x)z)=x-x+y(x+(-x)z)=x+(-x+y)z.
\end{equation}
Take $y=xx$ and $z=x$ in (\ref{lambda-antihomomorphic5}). Then $(x-x)xx(x+(-x)x)=x+(-x+xx)x$.
This  gives that
$$-x+x^{-1}(x+(-x+xx)x)=-x+x^{-1}(x-x)xx(x+(-x)x)$$$$=-x+x^{-1}xx^{-1}xx(x+(-x)x)=-x+x^{-1}xx(x+(-x)x)$$$$=-x+x^{-1}x+x(x+(-x)x)=x^{-1}x-x+x(x+(-x)x)
\overset{(\ref{lambda-antihomomorphic4})}{=}x^{-1}x+x=x^{-1}xx.$$
By the above statement, we have $x=x^{-1}xx$, whence $xx^{-1}=x^{-1}xxx^{-1}$. Substituting $x$ by $x^{-1}$ in this equation, we obtain $x^{-1}x=xx^{-1}x^{-1}x$.
This gives $x^{-1}x=xx^{-1}x^{-1}x=x^{-1}xxx^{-1}=xx^{-1}$. Hence $(S,\cdot)$ is also a Clifford semigroup and $(S,+, \cdot)$ is a dual weak left brace.
\end{proof}

From the above statements, we have known that  symmetric, $\lambda$-homomorphic and $\lambda$-anti-homomorphic weak left braces are all dual weak left braces.
To describe their structures, we need to recall a construction method of dual weak left braces  given in \cite{Catino-Mazzotta-Stefanelli3}.
Recall that a {\em lower semilattice} $(Y, \leq)$ is a partially ordered set  in which any two elements $\alpha$ and $\beta$ have a greatest lower bound $\alpha \beta$.
\begin{lemma}[Theorem 1 in \cite{Catino-Mazzotta-Stefanelli3}]\label{gouzao}
Let $Y$ be a lower semilattice, $\left\{(B_{\alpha}, +, \cdot)\ \left|\ \alpha \in Y\right.\right\}$ a family of disjoint skew left braces. For each pair $\alpha,\beta$ of elements of $Y$ such that $\alpha \geq \beta$, let $\phii{\alpha}{\beta}:B_{\alpha}\to B_{\beta}$ be a skew left brace homomorphism such that
\begin{enumerate}
    \item[(1)] $\phii{\alpha}{\alpha}$ is the identical automorphism of $B_{\alpha}$, for every $\alpha \in Y$,
    \item[(2)] $\phii{\beta}{\gamma}{}\phii{\alpha}{\beta}{} = \phii{\alpha}{\gamma}{}$, for all $\alpha, \beta, \gamma \in Y$ such that $\alpha \geq \beta \geq \gamma$.
\end{enumerate}
Then $S = \bigcup\left\{B_{\alpha}\ \left|\ \alpha\in Y\right.\right\}$ endowed with the addition and the multiplication, respectively, defined by
\begin{align*}
    a+b= \phii{\alpha}{\alpha\beta}(a)+\phii{\beta}{\alpha\beta}(b)
    \quad\text{and}\quad
     a b= \phii{\alpha}{\alpha\beta}(a) \phii{\beta}{\alpha\beta}(b),
\end{align*}
for all $a\in B_{\alpha}$ and $b\in B_{\beta}$, is a dual weak left brace, called \emph{strong semilattice $S$ of skew left braces $B_{\alpha}$} and denoted by $S=[Y; B_\alpha;\phii{\alpha}{\beta}]$.  Conversely, any dual weak left brace can be constructed in this way.
\end{lemma}
In view of  Propositions \ref{duichen}, \ref{tongtai}, \ref{fantongtai} and Lemma \ref{gouzao}, it is not hard to prove the following  proposition which gives the constructions of symmetric,  $\lambda$-homomorphic and$\lambda$-anti-homomorphic  weak left braces.
\begin{prop}\label{yy}
A symmetric (respectively, $\lambda$-homomorphic, $\lambda$-anti-homomorphic) weak left brace is exactly a strong semilattice of a family of symmetric (respectively, $\lambda$-homomorphic, $\lambda$-anti-homomorphic) skew left braces.
\end{prop}
Observe that \cite[Proposition 3.11]{Bardakov2} says that a skew left brace $(S, +, \cdot)$ is symmetric if and only if $(S, +, \cdot)$ is $\lambda$-anti-homomorphic. This together with  Proposition \ref{yy} yields the following result.
\begin{prop}
Symmetric weak left braces are exactly $\lambda$-anti-homomorphic weak left braces.
\end{prop}
\noindent {\bf Funding:}   The paper is supported by Nature Science Foundations of  China (12271442, 11661082).

\end{document}